\documentclass[letterpaper,11pt]{amsart}

\usepackage[labelfont=bf]{caption}
\usepackage{amsmath,amsfonts,amssymb,amsthm,url}
\usepackage{enumerate}
\usepackage{graphicx,natbib}

\newtheorem{theorem}{Theorem}

\newtheorem{proposition-appn}{Proposition~A\ignorespaces}
\newtheorem{lemma-appn}{Lemma~A\ignorespaces}
\newtheorem{theorem-appn}{Theorem~A\ignorespaces}

\theoremstyle{definition}
\newtheorem{remark}{Remark}

\let\oldproofname=\proofname
\renewcommand{\proofname}{\rm\bf{\oldproofname}}

\def\Pr{{ \mathrm{Pr} }}
\def\E{{ \mathrm{E} }}
\def\Var{{ \mathrm{Var} }}

\begin{document}

\title[]{Expectations of some ratio-type estimators under the gamma distribution}

\author{Jia-Han Shih}

\address{Department of Applied Mathematics, National Sun Yat-sen University, Kaohsiung, Taiwan}

\email{jhshih@math.nsysu.edu.tw} 

\keywords{Atkinson index, Dirichlet distribution, Theil index, Gini index, Proportion-sum independence theorem} 

\subjclass[2010]{60E10, 62E10, 62F10} 

\date{\today}

\begin{abstract}
We study the expectations of some ratio-type estimators under the gamma distribution. Expectations of ratio-type estimators are often difficult to compute due to the nature that they are constructed by combining two separate estimators. With the aid of Lukacs' Theorem and the gamma-beta (gamma-Dirichlet) relationship, we provide alternative proofs for the expected values of some common ratio-type estimators, including the sample Gini index, the sample Theil index, and the sample Atkinson index, under the gamma distribution. Our proofs using the distributional properties of the gamma distribution are much simpler than the existing ones. In addition, we also derive the expected value of the sample variance-to-mean ratio under the gamma distribution.
\end{abstract}

\maketitle 

\section{Introduction}\label{sec:intro}

Ratio-type statistical measures arise when two quantities are combined to form a scale-free population characteristic. Since the numerator and denominator are measured in the same units, their ratio is scale-invariant, resulting in a standardized quantity that is easy to interpret. If the goal is to produce a scale-free measure, a natural choice for the denominator is the population mean. For instance, the well-known Gini index \citep{gini1912} is a measure of income inequality, which is defined as the Gini mean difference divided by the population mean. For more details about the Gini index, we refer to Baydil et al.\ \citep{BAYDIL2025} and the references therein.

Estimation of ratio-type measures is typically carried out by estimating the numerator and denominator separately. If both components are estimated consistently, their ratio will also yield a consistent estimator for the population ratio thanks to the continuous mapping theorem. However, the finite-sample properties of a ratio-type estimator, such as its bias, can be challenging to derive due to the fact that the estimator is formed by combining two separate estimators. Recently, Baydil et al.\ \cite{BAYDIL2025} proved that the sample Gini index is actually unbiased under the gamma distribution by using the identity (\ref{id}) and the Laplace transform.


In this note, we give a simple alternative proof for the unbiasedness of the sample Gini index under the gamma distribution. The main tool of our proof is Lukacs' Theorem \citep{Lukacs1955}, which is stated below.

\begin{theorem}[Lukacs \citep{Lukacs1955}]
Let $X_1$ and $X_2$ be nondegenerate, independent, and positive random variables. Then the proportion $X_1 / (X_1 + X_2)$ and the sum $X_1+X_2$ are independent if and only if $X_1$ and $X_2$ are gamma random variables with the same scale (or rate) parameter.
\end{theorem}

Lukacs' Theorem shows that the independence between a proportion and the corresponding sum uniquely characterizes the gamma distribution. Applying Lukacs' Theorem along with the gamma-beta relationship, one can readily verify the unbiasedness of the sample Gini index. To further illustrate the usefulness of Lukacs' Theorem, we derive the expectation of the sample variance-to-mean ratio (VMR), which is also a ratio-type estimator. 

Besides the Gini index, there are other measures of income inequality, such as the Theil index and the Atkinson index \citep{theil1967,ATKINSON1970}, both of which are ratio-type measures. Recently, Vila and Saulo \cite{vila2025} employed the same approach as Baydil et al.\ \cite{BAYDIL2025} and computed the expected values of the sample Theil index and the sample Atkinson index under the gamma distribution. It turns that the gamma-Dirichlet relationship can be utilized for computing the expectation of the sample Theil index and the sample Atkinson index. 

The rest of this note is organized as follows: Section~\ref{sec-G} presents an alternative proof for the unbiasedness of the sample Gini index. Section~\ref{sec-T} gives alternative proofs for computing the expectation of the sample Theil index and the sample Atkinson index. Section~\ref{sec-VMR} derives the expectation of the sample VMR. Section~\ref{sec-conclude} concludes.

\section{The Gini index under the gamma distribution}\label{sec-G}

Recall that the gamma density is defined as
\begin{align}\label{gamma-density}
f(y) = \frac{\lambda^\alpha}{\Gamma(\alpha)} y^{\alpha-1} e^{-\lambda y}, \quad y > 0,
\end{align}
where $\Gamma (\cdot)$ is the gamma function, $\alpha > 0$ is the shape parameter, and $\lambda > 0$ is the rate parameter. Throughout this note, let $Y$ be a nondegenerate and positive random variable with mean $0 < \mu = \E (Y) < \infty$. The Gini index of $Y$ is defined as 
\begin{align*}
G(Y) = \frac{ \E | Y_1 - Y_2 | }{ 2 \mu },
\end{align*}
where $Y_1$ and $Y_2$ are two independent and identically distributed (i.i.d.) copies of $Y$. If $Y$ follows the gamma density in (\ref{gamma-density}), it has been shown that 
\begin{align}\label{G-gamma}
G(Y) = \frac{ \Gamma (\alpha + 1/2) }{ \sqrt{\pi} \; \Gamma(\alpha + 1) },
\end{align}
which does not depend on the rate parameter $\lambda$; see \citep{McDonald1979}.

Let $Y_1,\ldots,Y_n$ be i.i.d.\ random samples from the same population. The sample Gini index is given by
\begin{align}\label{G-est}
G_n = \frac{1}{n-1} \frac{ \sum_{1 \leq i < j \leq n} | Y_i - Y_j | }{ \sum_{i = 1}^n Y_i }, \quad n \geq 2.
\end{align}
It is obvious that $G_n$ is a consistent estimator for $G$, i.e., $G_n \to G$ in probability as $n \to \infty$. Ratio-type estimators are generally expected to be biased, but Baydil et al.\ \cite{BAYDIL2025} showed that $G_n$ is, in fact, an unbiased estimator for $G$ under the gamma distribution.

\begin{theorem}[Baydil et al., \citep{BAYDIL2025}]
Let $Y_1,\ldots,Y_n$ be i.i.d.\ copies of $Y$ which follows the gamma density given in (\ref{gamma-density}). Then $\E ( G_n ) = G(Y)$ in (\ref{G-gamma}).
\end{theorem}

As mentioned in Baydil et al.\ \cite{BAYDIL2025}, the main challenge in evaluating $\E (G_n)$ lies in the denominator in (\ref{G-est}). To handle this term, they rewrite the expectation as follows:
\begin{align*}
\E ( G_n )
&= \frac{1}{n-1} \E \left\{ \frac{ \sum_{1 \leq i < j \leq n} | Y_i - Y_j | }{ \sum_{i = 1}^n Y_i } \right\} \\
&= \frac{1}{n-1} \E \left\{ \sum_{1 \leq i < j \leq n} | Y_i - Y_j | \int_0^\infty e^{-z \sum_{i = 1}^n Y_i } dz \right\},
\end{align*}
where the last equality follows from the identity
\begin{align}\label{id}
\frac{1}{s} = \int_0^\infty e^{-z s} dz, \quad s > 0.
\end{align}
By using the Laplace transform of the gamma random variable along with some complicated integration, they managed to show that $\E ( G_n ) = G$, thus establishing that $G_n$ is an unbiased estimator for $G$ under the gamma distribution.

We now present an alternative proof for the unbiasedness of $G_n$ under the gamma distribution via Lukacs' Theorem and the gamma-beta relationship.

\begin{proof}
It suffices to consider
\begin{align*}
\E \left\{ \frac{ | Y_1 - Y_2 | }{ \sum_{i = 1}^n Y_i } \right\} = \E \left\{ \frac{Y_1 + Y_2}{ Y_1 + Y_2 + \sum_{i = 3}^n Y_i } \left| \frac{2 Y_1}{Y_1 + Y_2} - 1 \right| \right\}.
\end{align*}
First, note that the pair $( Y_1+Y_2, Y_1 / (Y_1 + Y_2) )$ and the vector $(Y_3,\ldots,Y_n)$ are independent. In addition, since $Y_1$ and $Y_2$ are i.i.d.\ gamma random variables, we also have the independence between $Y_1+Y_2$ and $Y_1 / (Y_1 + Y_2)$ according to Lukacs' Theorem \citep{Lukacs1955}. Therefore, $Y_1 + Y_2$, $Y_1 / (Y_1 + Y_2)$, and $(Y_3,\ldots,Y_n)$ are mutually independent. Consequently, the above expectation equals to
\begin{align*}
\E \left\{ \frac{Y_1 + Y_2}{ Y_1 + Y_2 + \sum_{i = 3}^n Y_i } \right\} \E \left\{ \left| \frac{2 Y_1}{Y_1 + Y_2} - 1 \right| \right\}.
\end{align*}
We now compute the expectations separately. Note that $(Y_1 + Y_2)/ \sum_{i = 1}^n Y_i$ follows the beta distribution with shape parameters $2 \alpha$ and $(n-2) \alpha$. We have
\begin{align*}
\E \left\{ \frac{Y_1 + Y_2}{ Y_1 + Y_2 + \sum_{i = 3}^n Y_i } \right\} = \frac{2}{n}.
\end{align*}
On the other hand, the random variable $R = Y_1 / (Y_1 + Y_2)$ follows the beta distribution with both shape parameters equal to $\alpha$. We have
\begin{align*}
\E | 2R - 1 |
&= \frac{1}{B(\alpha,\alpha)} \int_0^1 | 2r - 1 | r^{\alpha-1} (1-r)^{\alpha-1} dr \\
&= \frac{2}{B(\alpha,\alpha)} \int_{1/2}^1 (2r - 1) r^{\alpha-1} (1-r)^{\alpha-1} dr,
\end{align*}
where $B(\cdot,\cdot)$ is the beta function and the last equality follows from the symmetry of the integrand about $r = 1/2$. After making a change of variable $u = (2r-1)^2$, the integral transforms to
\begin{align*}
\E | 2R - 1 | = \frac{1}{2^{2\alpha-1} B(\alpha,\alpha)} \int_0^1 (1-u)^{\alpha-1} du
&= \frac{\Gamma(2\alpha)}{2^{2\alpha-1} \alpha \Gamma(\alpha)^2} \\
&= \frac{\Gamma(\alpha+1/2)}{\sqrt{\pi} \; \Gamma(\alpha+1)},
\end{align*}
where the last equality follows from the duplication formula \citep[Eq.\ (5.5.5)]{NIST:DLMF}, i.e., $\Gamma (\alpha) \Gamma (\alpha+1/2) = 2^{1-2\alpha} \sqrt{\pi} \Gamma (2\alpha) $, and the identity $\alpha \Gamma(\alpha) = \Gamma(\alpha+1)$. Combining all the results, we arrive at
\begin{align*}
\E ( G_n ) = \frac{1}{n-1} \E \left\{ \frac{ \sum_{1 \leq i < j \leq n} | Y_i - Y_j | }{ \sum_{i = 1}^n Y_i } \right\}
&= \frac{1}{n-1} \; \binom{n}{2} \; \E \left\{ \frac{ | Y_1 - Y_2 | }{ \sum_{i = 1}^n Y_i } \right\} \\
&= \frac{\Gamma(\alpha+1/2)}{\sqrt{\pi} \; \Gamma(\alpha+1)}.
\end{align*}
Hence $G_n$ is an unbiased estimator for $G$ under the gamma distribution.
\end{proof}

It is clear that our new proof relies on Lukacs' Theorem, i.e., the independence between $Y_1 + Y_2$ and $Y_1 / (Y_1 + Y_2)$, which is an unique characterization of the gamma distribution. In fact, the same approach can be used to compute the population Gini index $G(Y)$. Indeed, under the gamma distribution, 
\begin{align*}
\E | Y_1 - Y_2 | = \E ( Y_1 + Y_2 ) \E | 2R - 1 | = 2 \mu \frac{\Gamma(\alpha+1/2)}{\sqrt{\pi} \; \Gamma(\alpha+1)}.
\end{align*}
After cancelling out $2 \mu$, we obtain the population Gini index (\ref{G-gamma}).

\begin{remark}
Our proof does not imply that $\E ( G_n ) = G(Y)$ if and only if $Y$ follows the gamma distribution. It would be unrealistic to expect that a moment condition is able to enforce the underlying random variable to follow a specific distribution. For instance, one can directly verify that for a discrete random variable $Y$ satisfying
\begin{align*}
\Pr (Y = a) = \Pr (Y = b) = 1/2, \quad 0 < a < b < \infty,
\end{align*}
the sample Gini index $G_n$ is unbiased for $G(Y)$ when $n = 2$.
\end{remark}



\section{The Theil and Atkinson indices under the gamma distribution}\label{sec-T}

In this section, we apply the gamma-beta and the gamma-Dirichlet relationship to compute the sample Theil index and the sample Atkinson index.

\subsection{The Theil index}

The Theil $T$ index of $Y$ is defined as
\begin{align*}
T(Y) = \E \left\{ \frac{Y}{\mu} \log \left( \frac{Y}{\mu} \right) \right\}.
\end{align*}
If $Y$ follows the gamma density in (\ref{gamma-density}), it has been shown that 
\begin{align*}
T(Y) = \psi (\alpha) + \frac{1}{\alpha} - \log (\alpha),
\end{align*}
which does not depend on the rate parameter $\lambda$ and $\psi (\cdot)$ is the digamma function; see \citep{McDonald1979}. The sample Theil $T$ index is given by
\begin{align*}
T_n = \frac{\sum_{i = 1}^n Y_i \log ( Y_i / \mu_n )}{\sum_{i = 1}^n Y_i}, \quad n \geq 1,
\end{align*}
where $\mu_n = \sum_{i = 1}^n Y_i / n$ is the sample mean. Vila and Saulo \cite{vila2025} followed the same approach as Baydil et al.\ \cite{BAYDIL2025} to compute the expectation of $T_n$ under the gamma distribution. Here we state the result in Vila and Saulo \cite{vila2025} and provide a simple alternative proof.

\begin{theorem}[Vila and Saulo, \cite{vila2025}]
Let $Y_1,\ldots,Y_n$ be i.i.d.\ copies of $Y$ which follows the gamma density given in (\ref{gamma-density}). Then
\begin{align*}
\E ( T_n ) = \psi (\alpha) + \frac{1}{\alpha} + \log (n) - \psi (n \alpha) - \frac{1}{n \alpha}.
\end{align*}
\end{theorem}

\begin{proof}
It suffices to consider
\begin{align*}
\E \left\{ \frac{ Y_1 \log ( Y_1 / \mu_n )}{\sum_{i = 1}^n Y_i } \right\} = \E \left\{ \frac{ Y_1}{\sum_{i = 1}^n Y_i } \log \left( \frac{ Y_1}{\sum_{i = 1}^n Y_i } \right) \right\} + \E \left( \frac{ Y_1}{\sum_{i = 1}^n Y_i } \right) \log (n).
\end{align*}
Since the random variable $W = Y_1 / \sum_{i = 1}^n Y_i$ follows the beta distribution with shape parameters $\alpha$ and $(n-1) \alpha$, we only need to compute the expectation of $W \log (W)$. For a beta random variable $U$ with the shape parameters $a > 0$ and $b > 0$, we have
\begin{align}\label{exp-beta}
\E \{ U \log (U) \}
&= \frac{1}{B(a,b)} \int_0^1 u \log(u) u^{a-1} (1-u)^{b-1} du \nonumber \\
&= \frac{1}{B(a,b)} \int_0^1 \frac{\partial}{\partial a} u^{a} (1-u)^{b-1} du \nonumber \\
&= \frac{1}{B(a,b)} \frac{\partial}{\partial a} \int_0^1 u^{a} (1-u)^{b-1} du \nonumber \\
&= \frac{1}{B(a,b)} \frac{\partial}{\partial a} B(a+1,b) \nonumber \\
&= \frac{a}{a+b} \{ \psi (a+1) - \psi (a+b+1) \},
\end{align}
where the third equality holds due to the fact that the beta density belongs to the exponential family. Setting $a$ and $b$ as $\alpha$ and $(n-1)\alpha$ in (\ref{exp-beta}), respectively, we obtain
\begin{align*}
\E \{ W \log (W) \}
&= \frac{1}{n} \{ \psi (\alpha +1) - \psi (n \alpha +1) \} \\
&= \frac{1}{n} \left\{ \psi (\alpha) + \frac{1}{\alpha} - \psi (n \alpha) - \frac{1}{n\alpha} \right\},
\end{align*}
where $\psi (\alpha + 1) = \psi (\alpha) + 1 / \alpha$. Finally, we arrive at
\begin{align*}
\E ( T_n ) = n \E \left\{ \frac{ Y_1 \log ( Y_1 / \mu_n )}{\sum_{i = 1}^n Y_i } \right\} = \psi (\alpha) + \frac{1}{\alpha} + \log (n) - \psi (n \alpha) - \frac{1}{n \alpha}.
\end{align*}
Hence $T_n$ is biased under the gamma distribution.
\end{proof}

\subsection{The Atkinson index}

The Atkinson index of $Y$ is defined as
\begin{align*}
A(Y) = 1 - \exp \left[ - \E \left\{ \log \left( \frac{\mu}{Y} \right) \right\} \right],
\end{align*}
which may also be defined through the so-called Theil $L$ index \citep{theil1967,vila2025}. If $Y$ follows the gamma density in (\ref{gamma-density}), it has been shown that 
\begin{align*}
A(Y) = 1 - \frac{\exp \{ \psi (\alpha) \} }{\alpha},
\end{align*}
which does not depend on the rate parameter $\lambda$; see \citep{vila2025}. The sample Atkinson index is given by
\begin{align*}
A_n = 1 - \frac{ \left( \prod_{i = 1}^n Y_i \right)^{1/n}}{\mu_n}, \quad n \geq 1.
\end{align*}
In a similar fashion, we state the result of $\E (A_n)$ under the gamma distribution in Vila and Saulo \cite{vila2025} and provide a simple alternative proof.

\begin{theorem}[Vila and Saulo, \cite{vila2025}]
Let $Y_1,\ldots,Y_n$ be i.i.d.\ copies of $Y$ which follows the gamma density given in (\ref{gamma-density}). Then
\begin{align*}
\E ( A_n ) = 1 - \frac{\Gamma^n (\alpha + 1/n)}{\alpha \Gamma^n (\alpha)}.
\end{align*}
\end{theorem}

\begin{proof}
Define $Z_j = Y_j / \sum_{i = 1}^n Y_i$ for $j = 1,\ldots,n$. We rewrite
\begin{align*}
A_n = 1 - \frac{ \left( \prod_{i = 1}^n Y_i \right)^{1/n}}{\mu_n} = 1 - n \prod_{i = 1}^n \left( \frac{Y_i}{\sum_{i = 1}^n Y_i} \right)^{1/n} = 1 - n \prod_{i = 1}^n Z_i^{1/n}.
\end{align*}
Since the random vector $(Z_1,\ldots,Z_n)$ follows the Dirichlet distribution with all shape parameters equal to $\alpha$, by the product moment formula under the Dirichlet distribution \citep[Eq.\ (49.7)]{kotz2019}, we obtain
\begin{align*}
\E \left( \prod_{i = 1}^n Z_i^{1/n} \right) = \frac{\Gamma (n \alpha)}{\Gamma (n \alpha +1)} \frac{ \prod_{i = 1}^n \Gamma (\alpha + 1/n) }{ \prod_{i = 1}^n \Gamma (\alpha) } = \frac{\Gamma^n (\alpha + 1/n)}{n \alpha \Gamma^n (\alpha)}.
\end{align*}
Consequently,
\begin{align*}
\E (A_n) = 1 - \frac{\Gamma^n (\alpha + 1/n)}{\alpha \Gamma^n (\alpha)}.
\end{align*}
Hence $A_n$ is biased under the gamma distribution.
\end{proof}

\section{The variance-to-mean ratio under the gamma distribution}\label{sec-VMR}

The variance-to-mean ratio (VMR) of $Y$ is defined as 
\begin{align*}
\text{VMR}(Y) = \frac{\sigma^2}{\mu},
\end{align*}
provided that $\sigma^2 = \Var(Y) < \infty$. If $Y$ follows the gamma density in (\ref{gamma-density}), it is straightforward to obtain that
\begin{align*}
\text{VMR}(Y) = \frac{1}{\lambda}.
\end{align*}
The sample VMR is given by
\begin{align*}
\text{VMR}_n = \frac{1}{n-1} \frac{ \sum_{i = 1}^n ( Y_i - \mu_n )^2 }{ \mu_n }, \quad n \geq 2.
\end{align*}
Below we demonstrate that Lukacs' Theorem can be applied to compute the expectation of $\text{VMR}_n$.

\begin{theorem}\label{thm-vrm}
Let $Y_1,\ldots,Y_n$ be i.i.d.\ copies of $Y$ which follows the gamma density given in (\ref{gamma-density}). Then
\begin{align*}
\E({\mathrm{VMR}}_n) = \frac{n \alpha}{(n \alpha + 1) \lambda}.
\end{align*}
\end{theorem}

\begin{proof}
We begin by rewriting
\begin{align*}
\text{VMR}_n 
&= \frac{1}{n-1} \frac{ \sum_{i = 1}^n Y_i^2 - n \mu_n^2 }{ \mu_n } \\
&= \frac{n}{n-1} \sum_{i = 1}^n \left( \frac{Y_i}{\sum_{i = 1}^n Y_i} \right)^2 \left( \sum_{i = 1}^n Y_i \right) - \frac{1}{n-1} \sum_{i = 1}^n Y_i.
\end{align*}
By Lukacs' Theorem, $Y_j / \sum_{i = 1}^n Y_i$ is independent of $\sum_{i = 1}^n Y_i$. Moreover, $Y_j / \sum_{i = 1}^n Y_i$ follows the beta distribution with shape parameters $\alpha$ and $(n-1)\alpha$ for all $j = 1,\ldots,n$. Therefore, we have
\begin{align*}
\E \left\{ \left( \frac{Y_1}{\sum_{i = 1}^n Y_i} \right)^2 \left( \sum_{i = 1}^n Y_i \right) \right\}
&= \E \left( \frac{Y_1}{\sum_{i = 1}^n Y_i} \right)^2 \E \left( \sum_{i = 1}^n Y_i \right) \\
&= \left( \frac{n-1}{n^2(n \alpha +1)} + \frac{1}{n^2} \right) \frac{n \alpha}{\lambda}.
\end{align*}
Consequently,
\begin{align*}
\E ( \text{VMR}_n ) = \frac{n}{n-1} \left( \frac{n-1}{n^2(n \alpha +1)} + \frac{1}{n^2} \right) \frac{n \alpha}{\lambda} - \frac{n}{n-1} \frac{\alpha}{\lambda} = \frac{n \alpha}{(n \alpha + 1) \lambda}.
\end{align*}
\end{proof}

Theorem~\ref{thm-vrm} implies that if the random samples are drawn from the gamma distribution, the sample VMR tends to be biased downward:
\begin{align*}
\E ( \text{VMR}_n ) - \text{VMR}(Y) = \frac{n \alpha}{(n \alpha + 1) \lambda} - \frac{1}{\lambda} = -\frac{1}{(n \alpha + 1) \lambda} < 0.
\end{align*}

\section{Concluding remarks}\label{sec-conclude}

This note illustrates that the distributional properties of the gamma distribution, namely Lukacs' Theorem and the gamma-beta (gamma-Dirichlet) relationship, are extremely powerful for deriving the expected values of ratio-type estimators. Although we present the results only for measures of inequality, i.e., dispersion indices, the method may also be applied to other ratio-type statistical measures.

\section*{Acknowledgments}

The author thank Po-Han Hsu and Chee Han Tan for helpful discussions. J.-H.\ Shih is funded by National Science and Technology Council of Taiwan (NSTC 112-2118-M-110-004-MY2).

\section*{Declaration}

The author declare that there is no conflict of interest.

\bibliographystyle{plain}
\bibliography{paper-ref}

\end{document}